\documentclass[12pt]{amsart}
\usepackage{amssymb, amscd, amsmath, amsthm, 
epsf, epsfig, latexsym, enumerate}

\newtheorem{theorem}{Theorem}
\newtheorem{lemma}[theorem]{Lemma}
\newtheorem{cor}[theorem]{Corollary}

\newtheorem*{thm}{Theorem}

\begin{document}
\title{Indecomposable non-orientable $PD_3$-complexes}

\author{J.A. Hillman }
\address{School of Mathematics and Statistics, University of Sydney,
 \newline
 NSW 2006,  Australia}
 
\email{jonathan.hillman@sydney.edu.au}



\begin{abstract}
We show that the orientable double covering space of an indecomposable,
non-orientable $PD_3$-complex has torsion free fundamental group.
\end{abstract}

\maketitle

One of the foundational results of Wall on Poincar\'e duality complexes
was the fact that there is a well defined notion of connected sum for such complexes \cite{wall}.
In dimensions $n>2$ the fundamental group of a connected sum of two $PD_n$-complexes
is the free product of the groups of the summands.
This notion is of particular interest when $n=3$, 
for by the well-known work of Kneser and Milnor
every closed orientable 3-manifold has an essentially unique 
factorization into indecomposable 3-manifolds.
(The corresponding assertion for closed non-orientable 3-manifolds 
is slightly more complicated.)
Moreover such a 3-manifold is indecomposable with respect to 
connected sum if and only if its fundamental group is
indecomposable with respect to free product.
It is perhaps less widely known that Tura'ev has shown that 
each of these results extends to the
context of $PD_3$-complexes \cite{tu}.

Indecomposable orientable 3-manifolds are either aspherical, 
have finite fundamental group or have fundamental group $\mathbb{Z}$.
This is no longer true for $PD_3$-complexes, 
although Crisp has shown that (in the orientable case) the indecomposables 
are either aspherical or have virtually free fundamental group \cite{Cr}.
There are examples of the latter kind with fundamental group neither finite nor 
$\mathbb{Z}$ \cite{Hi12}.

Let $X$ be an indecomposable $PD_3$-complex, 
with fundamental group $\pi$ and orientation character $w$.
In \cite{Hi12} we showed that if $w\not=1$ and $\pi$ is virtually free then 
$X$ is homotopy equivalent to $S^2\tilde\times{S^1}$ or $RP^2\times{S^1}$,
and so $\pi\cong\mathbb{Z}$ or $\mathbb{Z}\oplus{\mathbb{Z}/2\mathbb{Z}}$.
In particular, $\pi^+=\mathrm{Ker}(w)$ is torsion free.
We shall show that this remains true if $w\not=1$ and $\pi$ is {\it not\/} virtually free.
This result is surely well-known for 3-manifolds.
We give a short proof for this case in \S1,
which uses the ``projective plane theorem" of \cite{Ep} and a result from \cite{Hi12}.
(The fact that $RP^2$ does not bound provides a further restriction here 
which is not yet known in general.)
Our main result is Theorem 5 in \S2:

{\it Let $X$ be an indecomposable, non-orientable $PD_3$-complex
such that $\pi$ has infinitely many ends.
Then $\pi\cong\pi^+\rtimes{\mathbb{Z}/2\mathbb{Z}^-}$,
and $\pi^+$ is torsion free, but not free.}

\noindent{By} passing to Sylow subgroups of the torsion in $\pi$, 
we may reduce potential counter-examples to special cases, 
which are eliminated by Lemmas 3 and 4. 
The arguments are similar to those of \cite{Hi12}.

I would like to thank B.Hanke for alerting me to the necessity of considering the present case.
I would also like to thank the referee for detailed suggestions as to improving the exposition 
of this work.

\section{notation and major cited results}

In order that this paper be reasonably self-contained we shall 
give here some of the notation and results used in \cite{Hi12}.

Let $X$ be a $PD_3$-complex, 
with fundamental group $\pi$ and orientation character $w$,
and let $X^+$ be the orientable covering space, with fundamental group $\pi^+=\mathrm{Ker}(w)$.
If $H\leq\pi$ then we shall write $H^+=H\cap\pi^+$.
It is convenient to say that such a subgroup $H$ is {\it orientable\/} if $H=H^+$.
(This usage depends upon the orientation character $w$.)
Let $\mathbb{Z}/2\mathbb{Z}^-$ denote a subgroup of order two on which $w\not=1$.

If $G$ is a group $|G|$, $G'$ and $\zeta{G}$ shall denote the order, 
commutator subgroup and centre of $G$, while if $H\leq{G}$ then $C_G(H)$
and $N_G(H)$ are the centralizer and normalizer, respectively.
Let $F(r)$ be the free group of rank $r$.

If $R$ is a ring two finitely presentable left $R$-modules 
$M$ and $N$ are {\it stably isomorphic\/} 
if $M_1\oplus{R}^a\cong{N}\oplus{R}^b$ for some $a,b\geq0$.
Let $[M]$ denote the stable isomorphism class of $M$.

A homomorphism $w:G\to\{\pm1\}$ defines an anti-involution of $\mathbb{Z}[G]$
by $\bar{g}=w(g)g^{-1}$, for all $g\in{G}$.
Tietze move considerations show that if $A$ is any finite presentation matrix for the augmentation ideal $I_G$ then the stable isomorphism class 
of the left $\mathbb{Z}[G]$-module $J_G$ with presentation matrix 
the conjugate transpose $\overline{A}^{tr}$ is well-defined \cite{tu}.

A {\it graph of groups} $(\mathcal{G},\Gamma)$ consists of a graph $\Gamma$ 
with origin and target functions $o$ and $t$ from the set of edges $E(\Gamma)$ 
to the set of vertices $V(\Gamma)$, and a family $\mathcal{G}$ of groups $G_v$ 
for each vertex $v$ and subgroups $G_e\leq{G_{o(e)}}$ for each edge $e$,
with monomorphisms $\phi_e:G_e\to{G_{t(e)}}$.
(We shall usually suppress the maps $\phi_e$ from our notation.)
In considering paths in $\Gamma$ we shall not require that 
the edges be compatibly oriented.

The {\it fundamental group} of $(\mathcal{G},\Gamma)$ is the group 
$\pi\mathcal{G}$ 
with presentation
\[\langle G_v,t_e\mid~t_egt_e^{-1}=\phi_e(g)~\forall{g}\in{G_e},~t_e=1~\forall{e}\in{E(T)}\rangle,\]
where $T$ is some maximal tree for $\Gamma$.
Different choices of maximal tree give isomorphic groups.
We may assume that $(\mathcal{G},\Gamma)$ is {\it reduced}: 
if an edge joins distinct vertices then the edge group is isomorphic to 
a proper subgroup of each of these vertex groups.
The corresponding $\pi$-tree $T$ is incompressible in the terminology of \cite{DD},
and so $T$ and $\mathcal{G}$ are essentially unique, by Proposition IV.7.4 of \cite{DD}.
An edge $e$ is a {\it loop isomorphism} at $v$ if $o(e)=t(e)=v$ 
and the inclusions induce isomorphisms $G_e\cong{G_v}$.

Since fundamental groups of $PD_n$-complexes are $FP_2$ \cite{wall}, 
$\pi$ is the fundamental group of a finite graph of groups $(\mathcal{G},\Gamma)$,
where all vertex groups are finite or have one end and all edge groups are finite.
(See Theorem VI.6.3 of \cite{DD}.)
We may assume that $\pi$ is indecomposable as a proper free product, 
by the Splitting Theorem, and so $(\mathcal{G},\Gamma)$ is {\it indecomposable}: 
all edge groups are nontrivial.
A graph of groups  $(\mathcal{G},\Gamma)$ is {\it admissible\/} if it is reduced, 
all vertex groups are finite or one-ended groups and all edge groups are
nontrivial finite groups.

Turaev gave the following characterization of the group-pairs $(\pi,w)$ 
which may be realized by finite $PD_3$--complexes \cite{tu}.

\begin{thm}
Let $\pi$ be a finitely presentable group and $w:\pi\to\{\pm1\}$ 
a homomorphism.
Then there is a finite $PD_3$-complex $K$ with $\pi_1(K)\cong\pi$ and 
$w_1(K)=w$ if and only if $[I_\pi]=[J_\pi]$.
\end{thm}

We wish to adapt the results from \S7 of \cite{Hi12} to the cases when $\pi$ 
has infinitely many ends and $w\not=1$.
In particular, we use the following two results to
to control the possible edge groups.
\begin{enumerate}
\item(Crisp's Theorem)
{\sl If $X$ is a $PD_3$-complex and $g\in\pi=\pi_1(X)$ has prime order $p$ 
and infinite centralizer $C_\pi(g)$ then $p=2$, 
$g$ is orientation-reversing and $C_\pi(g)$ has two ends.}
\item(the normalizer condition)
{\sl
a proper subgroup of a nilpotent group
is properly contained in its normalizer.}
\end{enumerate}
These are Theorem 17 of \cite{Cr} and Proposition 5.4.2 of \cite{rob}, respectively.
Note also that if $G$ is a finite subgroup of $\pi$ then the centralizer $C_\pi(G)$ has finite index in the 
normalizer $N_\pi(G)$.

The main result (Theorem 6 below) involves consideration of the
finite groups with periodic cohomology, of period dividing 4.
A finite group has cohomological period 2 if and only if it is cyclic, 
and has cohomological period 4 if and only if it is a product $B\times{\mathbb{Z}/d\mathbb{Z}}$ with $(|B|,d)=1$,
where $B$ is a generalized quaternionic group 
${\mathbb{Z}/a\mathbb{Z}\rtimes{Q(2^i)}}$ (with $a$ odd), 
an extended binary polyhedral group $T^*_k$
(of order $2^3.3^k$), $O^*_k$ (of order $2^4.3^k$)
or $I^*=SL(2,5)$ (of order $2^3.3.5$)
or a metacyclic group $\mathbb{Z}/a\mathbb{Z}\rtimes_{-1}{\mathbb{Z}/2^e\mathbb{Z}}$ (for some
odd $a$ and $e\ge1$). 

[There seems to be no one reference with a complete proof of the above assertion.
The six families of finite groups with periodic cohomology are
determined in pages 142-150 of \cite{AM}:
\begin{enumerate}
\item $\mathbb{Z}/a\mathbb{Z}\rtimes{\mathbb{Z}/b\mathbb{Z}}$;

\item $\mathbb{Z}/a\mathbb{Z}\rtimes(\mathbb{Z}/b\mathbb{Z}\times{Q(2^i)})$, $i\geq3$;

\item $\mathbb{Z}/a\mathbb{Z}\rtimes(\mathbb{Z}/b\mathbb{Z}\times{T^*_k})$, $k\geq1$;

\item $\mathbb{Z}/a\mathbb{Z}\rtimes(\mathbb{Z}/b\mathbb{Z}\times{O^*_k})$, $k\geq1$;

\item $(\mathbb{Z}/a\mathbb{Z}\rtimes{\mathbb{Z}/b\mathbb{Z}})\times{SL(2,p)}$, $p\geq5$ prime;

\item $\mathbb{Z}/a\mathbb{Z}\rtimes(\mathbb{Z}/b\mathbb{Z}\times{TL(2,p)})$, $p\geq5$ prime.

\end{enumerate}
Here $a$, $b$ and the order of the quotient by the metacyclic normal subgroup
$\mathbb{Z}/a\mathbb{Z}\rtimes{\mathbb{Z}/b\mathbb{Z}}$ are relatively prime.
See pages 142--150 of \cite {AM} for further details on the groups $TL(2,p)$ (with $TL(2,p)'\cong{SL(2,p)}$, of index 2) 
and the actions in the semidirect products.
If such a group $G$ contains a semidirect product $\mathbb{Z}/m\mathbb{Z}\rtimes_\theta{\mathbb{Z}/n\mathbb{Z}}$, 
where $\theta$ has image of order $k$, 
then the cohomological period of $G$ is a multiple of $2k$.
(See Exercise 6 on page 159 of \cite{Br}.)
The class of groups of period dividing 4 follows on applying this criterion to the groups of the above list.]

\section{3-manifolds}

The result is relatively easy (and no doubt well-known)
in the case of irreducible  3-manifolds,
as we may use the Sphere Theorem, as strengthened by Epstein \cite{Ep}.

\begin{theorem}
Let $M$ be an indecomposable, non-orientable $3$-manifold 
with fundamental group $\pi$.
If $\pi$ has infinitely many ends then $\pi\cong\pi^+\rtimes{\mathbb{Z}/2\mathbb{Z}^-}$
and $\pi^+$ is torsion free, but not free.
\end{theorem}

\begin{proof}
Let $\mathcal{P}$ be a maximal set of pairwise non-parallel 2-sided projective planes in $M$.
Then $\mathcal{P}$ is nonempty, since $M$ is indecomposable and $\pi$ has infinitely many ends.
In particular, $\pi\cong\pi^+\rtimes{\mathbb{Z}/2\mathbb{Z}^-}$, 
since the inclusion of a member of $\mathcal{P}$ splits $w=w_1(M):\pi\to{\mathbb{Z}/2\mathbb{Z}}$.
Let $\mathcal{P}^+$ be the preimage of $\mathcal{P}$ in $M^+$.
Then $\mathcal{P}^+$ is a set of disjoint 2-spheres in $M^+$,
and the components of $M^+\setminus\mathcal{P}^+$  each double cover a
component  of $M\setminus\mathcal{P}$.
Each such component of $M\setminus\mathcal{P}$ is indecomposable \cite{Ep}.

Suppose that $M\setminus\mathcal{P}$ has a component $Y$ 
with virtually free fundamental group.
Then the double $DY$ is indecomposable (cf. Lemma 2.4 of \cite{Hi12}), 
non-orientable and $\pi_1(DY)$ is virtually free.
Moreover, $\pi_1(DY)\cong\mathbb{Z}\oplus{\mathbb{Z}/2\mathbb{Z}^-}$,
since the inclusion of a boundary component of $Y$ splits $w$.
(See Theorems 7.1 and 7.4 of \cite{Hi12}.)
But then $DY\cong{RP^2\times{S^1}}$, and so $Y\cong{RP^2\times[0,1]}$.
This is contrary to the hypothesis that the members of $\mathcal{P}$ are non-parallel.
Thus the components of $M\setminus\mathcal{P}$ are punctured aspherical 3-manifolds.

Let $\Gamma$ be the graph with vertex set $\pi_0(M\setminus\mathcal{P})$ and edge set
$\mathcal{P}$, with an edge joining contiguous components.
Then $\pi^+\cong{G}*F(s)$,
where $G$ is a free product of $PD_3$-groups (corresponding to
the fundamental groups of the components of $M\setminus\mathcal{P}$),
and $s=\beta_1(\Gamma)$.
Hence $\pi^+$ is torsion free.
\end{proof}

We remark also that each component $Y$ of $M\setminus\mathcal{P}$ has an even number 
of boundary components, since $\chi(\partial{Y})$ is even
(for any odd-dimensional manifold $Y$), by Poincar\'e duality.
Thus the vertices of the graph $\Gamma$ have even valence.

{\it Example.}
The canonical involution $\iota$ of the topological group $T^3=\mathbb{R}^3/\mathbb{Z}^3$ 
has 8 isolated fixed points (the points of order 2). 
Let $X$ be the complement of an equivariant open regular neighbourhood of the fixed point set,
and let $M=D(X/\langle\iota\rangle)$.
Then $M$ is indecomposable and non-orientable, and $\pi\cong(\mathbb{Z}^3*\mathbb{Z}^3*F(7))\rtimes\mathbb{Z}/2\mathbb{Z}^-$.

\section{$PD_3$-complexes}

Suppose now that $X$ is an indecomposable $PD_3$-complex, 
with fundamental group $\pi$ and orientation character  $w$.
Then $\pi$ is finitely presentable, and so $\pi\cong\pi\mathcal{G}$,
where $(\mathcal{G},\Gamma)$ is an admissible graph of groups.

\begin{lemma}
Let $X$ be an indecomposable, non-orientable $PD_3$-complex with
$\pi=\pi_1(X)\cong\pi\mathcal{G}$,
where $(\mathcal{G},\Gamma)$ is an admissible graph of groups.
\begin{enumerate}
\item
if $e$ is an edge with $G_{o(e)}$ or $G_{t(e)}$ infinite then $G_e=\mathbb{Z}/2\mathbb{Z}^-$;
\item
if $X\not\simeq{S^2}\tilde\times{S^1}$ then $\pi\cong\pi^+\rtimes\mathbb{Z}/2\mathbb{Z}^-$;
\item
if all finite vertex groups are $2$-groups then they are non-orientable
and all edge groups are $\mathbb{Z}/2\mathbb{Z}^-$.
\end{enumerate}
\end{lemma}

\begin{proof}
Suppose first that the vertex groups are all finite.
Then $X\simeq{S^2}\tilde\times{S^1}$ (if all the vertex groups are orientation preserving) 
or $RP^2\times{S^1}$ (otherwise), by  Theorems 7.1 and 7.4 of \cite{Hi12}, respectively, 
and so the lemma holds.
Hence we may assume that $(\mathcal{G},\Gamma)$ has at least one infinite vertex group 
$G_v$ and at least one edge $e$ with $o(e)=v$ or $t(e)=v$.
If $w(g)=1$ for some $g\in{G_e}$ of prime order then both $G_{o(e)}^+$ and $G_{t(e)}^+$ would be finite, by
Theorem 14 of \cite{Cr}.
But then $G_v$ would be finite, contrary to hypothesis.
Thus $G_e=\mathbb{Z}/2\mathbb{Z}^-$,
and the inclusion of $G_e$ into $\pi$ splits $w$,
so $\pi\cong\pi^+\rtimes{\mathbb{Z}/2\mathbb{Z}^-}$.

Suppose that all finite subgroups are 2-groups.
Let  $f$ be an edge such that the vertex groups
$G_{o(f)}$ and $G_{t(f)}$ are finite.
If $G_f=G_{o(f)}$ (or $G_{t(f)}$) then $f$ must be a loop isomorphism,
since  $(\mathcal{G},\Gamma)$ is reduced.
But then $C_\pi(G_f)$ is infinite, and so $G_f=\mathbb{Z}/2\mathbb{Z}^-$,
by Crisp's Theorem.
Since $(\mathcal{G},\Gamma)$ is reduced, $f$ must be the only edge, 
contrary to the assumption that there is an infinite vertex group.
Thus we may assume that $G_{o(f)}$ and $G_{t(f)}$ each properly contain $G_f$. 
Since $G_{o(f)}$ and $G_{t(f)}$ are 2-groups and hence nilpotent, 
$N_\pi(G_f)$ is infinite, by  the normalizer condition.
Since $C_\pi(G_f)$ has finite index in $N_\pi(G_e)$ we must 
have $G_f=\mathbb{Z}/2\mathbb{Z}^-$, by Crisp's Theorem.
Since $\Gamma$ is connected it follows easily that every finite vertex group is non-orientable 
and every edge group is $\mathbb{Z}/2\mathbb{Z}^-$.
\end{proof}

The next two lemmas consider two parallel special cases,
involving a prime $p$, which is odd or 2, respectively.

\begin{lemma}
Let $X$ be an indecomposable $PD_3$-complex with $\pi=\pi_1(X)$ $\cong\kappa\rtimes{W}$,
where $\kappa$ is orientable and torsion free, and $W$ has order $2p$, for some odd prime $p$.
Then $X$ is orientable.
\end{lemma}

\begin{proof}
Suppose that $X$ is not orientable.
Then $\pi$ and $\kappa$ are infinite.
Since $\pi$ has a subgroup $W$ of finite order $>2$,  
we may assume that $\pi\cong\pi\mathcal{G}$, 
where $(\mathcal{G},\Gamma)$ is an admissible graph of 
groups with $r\geq1$ finite vertex groups and at least one edge.
Let $s=\beta_1(\Gamma)$.

Each finite vertex group is mapped injectively by any projection from $\pi$ onto $W$ with kernel $\kappa$.
If a vertex group $G_v$ has prime order then every edge $e$ with one vertex at $v$ 
is a loop isomorphism, since $(\mathcal{G},\Gamma)$ is reduced.
But then $\Gamma$ has just one vertex and $\pi\cong{G_v}\rtimes{F}$, 
which contradicts the hypothesis.
Hence all finite vertex groups are isomorphic to $W$.
If an edge $e$ is a loop isomorphism then $G_e^+\cong\mathbb{Z}/p\mathbb{Z}$ has infinite normalizer,
contradicting Crisp's Theorem.
If there is an edge $e$ with $G_e$ of order $p$ then both 
of the vertex groups  $G_{o(e)}$ and $G_{t(e)}$ are finite, by Lemma 2.
But then $[G_{o(e)}:G_e]=[G_{t(e)}:G_e]=2$, and so $N_\pi(G_e)$ is infinite,
which again contradicts Crisp's Theorem.
Since the orientation character $w$ factors through $W$ it follows that 
every edge group is $\mathbb{Z}/2\mathbb{Z}^-$ and $w$ is nontrivial on every vertex group.

Since each edge group is $\mathbb{Z}/2\mathbb{Z}^-$,
$w$ is nontrivial on each vertex group and so $\pi^+=\pi\mathcal{G}^+$
is the fundamental group of a graph of groups $(\mathcal{G}^+,\Gamma)$ with the same underlying graph $\Gamma$,
trivial edge groups and vertex groups $G_v^+$, for all $v\in{V(\Gamma)}$.
Hence $\pi^+\cong{G*F(s)*P}$, where $G$ is a free product of orientable $PD_3$-groups
and $P$ is a free product of $r$ copies of $\mathbb{Z}/p\mathbb{Z}$.
We have $P\cong{F(t)\rtimes\mathbb{Z}/p\mathbb{Z}}$ for some $t\geq0$.
(In fact, $t=(p-1)(r-1)$, by a simple
virtual Euler characteristic argument.)

Let $a\in\pi$ be such that $a^2=1$ and $w(a)=-1$,
and let $\lambda\cong\kappa\rtimes{\mathbb{Z}/2\mathbb{Z}^-}$ be the subgroup generated by $\kappa$ and $a$.
Then $\lambda$ is also the group of a $PD_3$-complex, since it has finite index in $\pi$.
The involution of $\pi^+$ induced by conjugation by $a$ maps each 
indecomposable factor which is not infinite cyclic to a conjugate of an isomorphic factor \cite{Gi}.
However, its behaviour on the free factor $F(s)$ may be more complicated.

Let $w:\mathbb{Z}[\pi]\to{R}=\mathbb{Z}[\langle{a}\rangle]=\mathbb{Z}[a]/(a^2-1)$
 be the linear extension of the orientation character.
Then $I_{\langle{a}\rangle}\cong\widetilde{\mathbb{Z}}=R/(a+1)$.
We may factor out the action of $\pi^+$ on a $\mathbb{Z}[\pi]$-module
by tensoring with $R$.
The derived sequence of the functor $R\otimes_w-$ applied to the augmentation sequence
\[
0\to{I_\pi}\to\mathbb{Z}[\pi]\to\mathbb{Z}
\]
 gives an exact sequence
\[
0\to{H_1(\pi;R)}=\kappa/\kappa'\to{R}\otimes_wI_\pi\to{R}\to\mathbb{Z}\to0.
\]
The inclusion of $\langle{a}\rangle$ into $\pi$ splits the epimorphism from
${R}\otimes_wI_\pi$ onto $I_{\langle{a}\rangle}$, and so
$R\otimes_wI_\pi\cong\kappa/\kappa'\oplus\widetilde{\mathbb{Z}}.$

Let $\gamma$ be the normal subgroup of $\pi$ generated by $G\cup{F(s)}$,
and let $H$ be the image of $\gamma$ in $\kappa/\kappa'$.
Then similar arguments show that
\[
R\otimes_wI_\pi=H\oplus(R\otimes_wI_{\pi/\gamma})
\]
 and
\[
R\otimes_wI_\lambda=H\oplus(R\otimes_wI_{\lambda/\gamma}).
\]

The groups $P$ and its normal subgroup $F(t)$ have presentations
\[
P=\langle{b_i,~1\leq{i}\leq{r}}\mid{b_i^p=1},~\forall{i}\rangle
\]
and
\[
F(t)=\langle{x_{i,j}, ~1\leq{i}\leq{r-1},~1\leq{j}\leq{p-1}}\mid~\rangle,
\]
where $x_{i,j}$ has image $b_1^jb_{i+1}^{-j}$ in $P$, 
for $1\leq{i}\leq{r-1}$ and $1\leq{j}\leq{p-1}$.
(If $p=2$ we shall write $x_i$ instead of $x_{i,1}$, for $1\leq{i}\leq{r-1}$.)

The quotient $\pi/\langle\langle{G}\rangle\rangle$ is the fundamental group 
of the (possibly {\it un}reduced) graph of groups  $(\overline{\mathcal{G}},\Gamma)$
with vertex groups $W$ (or $\mathbb{Z}/2\mathbb{Z}^-$) and edge groups $\mathbb{Z}/2\mathbb{Z}^-$,
obtained by replacing each infinite vertex group $G_v$  of $(\mathcal{G},\Gamma)$ 
by $G_v/G_v^+=\mathbb{Z}/2\mathbb{Z}^-$.
Thus if $W$ is abelian (and so has an unique element of order 2) then
$\pi/\langle\langle{G}\rangle\rangle\cong(F(s)*P)\times{\mathbb{Z}/2\mathbb{Z}^-}$.
Hence $\pi/\gamma\cong{P}\times\mathbb{Z}/2\mathbb{Z}^-$ and
$\lambda/\gamma\cong{F(t)}\times\mathbb{Z}/2\mathbb{Z}^-$,
and so
\[R\otimes_wI_{\pi/\gamma}\cong(R/(p,a-1))^r\oplus\widetilde{\mathbb{Z}}
\]
 and 
\[
R\otimes_wI_{\lambda/\gamma}\cong(R/(a-1))^t\oplus\widetilde{\mathbb{Z}}=\mathbb{Z}^t\oplus\widetilde{\mathbb{Z}}.
\]
The quotient ring $R/pR=\mathbb{F}_p[a]/(a^2-1)$ is semisimple, 
and so $p$-torsion $R$-modules have unique factorizations as sums of simple modules.
Since $I_\pi\otimes_w{R}$ and $I_\lambda\otimes_w{R}$ satisfy Turaev's criterion
(and projective $R$-modules are $\mathbb{Z}$-torsion free),  
the $p$-torsion submodule of $R\otimes_wI_\pi$ has equally many summands 
of types ${R/(p,a-1)}$ and $R/(p,a+1)$,
and similarly for $R\otimes_wI_\lambda$.
Since $R\otimes_wI_{\lambda/\gamma}$ is $p$-torsion free,
the number of summands of types $R/(p,a-1)$ and $R/(p,a+1)$ in $H$ must also be equal.  
On the other hand, $R\otimes_wI_{\pi/\gamma}$ has $r>0$ summands of type $R/(p,a-1)$
and none of type $R/(p,a+1)$.
These conditions are inconsistent, and so $\pi$ is not the group of a non-orientable $PD_3$-complex.

If $W$ is not abelian then has an unique conjugacy class of elements of order 2, and
 $\pi/\gamma\cong{P\rtimes{\mathbb{Z}/2\mathbb{Z}^-}}$ and $\lambda/\gamma\cong{F(t)}\rtimes{\mathbb{Z}/2\mathbb{Z}^-}$
have presentations
\[
\langle{a,b_i,~1\leq{i}\leq{r}}\mid{a^2=1},~b_i^p=1,~ab_ia=b_i^{-1}~\forall{i}\rangle,
\]
and
\[
\langle{a,x_{i,j},~1\leq{i}\leq{r-1},~1\leq{j}\leq{p-1}}\mid{a^2=1},
~ax_{ij}a=x_{i,p-j} ~\forall{i,j}\rangle,
\]
respectively. (In particular,  $\lambda/\gamma\cong{F(t/2)*\mathbb{Z}/2\mathbb{Z}^-}$.)
In this case 
\[R\otimes_wI_{\pi/\gamma}\cong(R/(p,a+1))^r\oplus\widetilde{\mathbb{Z}}
\] and 
\[
R\otimes_wI_{\lambda/\gamma}\cong{R^{t/2}}\oplus\widetilde{\mathbb{Z}}.
\]
Consideration of the $p$-torsion submodules again shows that  $R\otimes_wI_\pi$ 
and $R\otimes_wI_\lambda$ cannot both satisfy Turaev's criterion, 
and hence that $\pi$ is not the group of a non-orientable $PD_3$-complex.
Thus $X$ must be orientable.
\end{proof}

The case when the prime $p=2$ involves slightly different calculations.

\begin{lemma}
Let $X$ be an indecomposable $PD_3$-complex with $\pi=\pi_1(X)$ $\cong\kappa\rtimes{W}$,
where $\kappa$ is orientable and torsion free, and $W$ has order $4$. 
Then $X$ is orientable.
\end{lemma}

\begin{proof}
As in Lemma 3, we suppose that $X$ is not orientable, so $\pi$ and $\kappa$ are infinite,
and may assume that $\pi\cong\pi\mathcal{G}$, 
where $(\mathcal{G},\Gamma)$ is an admissible graph of 
groups with $r\geq1$ finite vertex groups and at least one edge.
We continue the notation $P$, $\gamma$, $a$ and $R$ from Lemma 3.

The inclusions of the edge groups split $w$, by Lemma 2.
In this case $W\cong(\mathbb{Z}/2\mathbb{Z})^2=\mathbb{Z}/2\mathbb{Z}\oplus{\mathbb{Z}/2\mathbb{Z}^-}$
and has two orientation reversing elements.
Note that  $P$ is now a free product of $r$ copies of $\mathbb{Z}/2\mathbb{Z}$.

The quotient $\pi/\gamma$ is the group of a finite graph of groups with all vertex groups $W$ 
and edge groups $\mathbb{Z}/2\mathbb{Z}^-$. 
Since $P$ is a free product of cyclic groups, $\pi/\gamma$ has a presentation
\[
\langle{a, b_i,~ 1\leq{i}\leq{r}}\mid{a^2=1},~b_i^2=(aw_i)^2=(aw_ib_i)^2=1,~\forall{i}\rangle,
\]
where $w_i=1$ and $w_i\in{F(t)}$ for  $2\leq{i}\leq{r}$.
The free subgroup $F(t)$ has basis $\{x_i|1\leq{i}\leq{r-1}\}$, 
where $x_i$ has image $b_1b_{i+1}$ in $P$,
and $\lambda/\gamma$ has a presentation
\[
\langle{a,x_i,~1\leq{i}\leq{r-1}}\mid{a^2=1},~ax_ia=x_ib_{i+1}w_{i+1}b_{i+1}w_{i+1}^{-1},~\forall{i}\rangle.
\]
In this case 
\[R\otimes_wI_{\pi/\gamma}\cong(R/(2,a-1))^r\oplus\widetilde{\mathbb{Z}}
\] and 
\[
R\otimes_wI_{\lambda/\gamma}\cong\mathbb{Z}^{r-1}\oplus\widetilde{\mathbb{Z}}.
\]
Since $R/(2,a+1)=R/(2,a-1)$, torsion considerations do not appear to help.
If $r>1$ we may instead compare the quotients by the 
$\mathbb{Z}$-torsion submodules, as in Lemma 7.3 of \cite{Hi12}, 
since finitely generated torsion free $R$-modules are direct sums of copies of $R$,
$\mathbb{Z}$ and $\widetilde{\mathbb{Z}}$,
by Theorem 74.3 of \cite{CR}.
We again conclude that $\pi$ is not the group of a non-orientable $PD_3$-complex.

The case when $p=2$ and $r=1$ requires a little more work.
Let $N$ be the $R$-module presented by the transposed conjugate of
$\left(\begin{smallmatrix}
2\\ 
a-1
\end{smallmatrix}\right)$.
If $\{e,f\}$ is the standard basis for $R^2$ then $N={R^2/R(2e+(a+1)f)}$.
The $\mathbb{Z}$-torsion submodule of $N$ is generated by the image of $(a-1)e$, 
and has order 2, but is not a direct summand.
The quotient of $N$ by its $ \mathbb{Z}$-torsion submodule is generated by the images 
of $e$ and $f-e$, and is a direct sum $\mathbb{Z}\oplus\widetilde{\mathbb{Z}}$.
In particular, it has no free summand.
It now follows easily that $H\oplus\widetilde{\mathbb{Z}}\oplus{R/(2,a-1)}$ is not 
stably isomorphic to $H\oplus\widetilde{\mathbb{Z}}\oplus{N}$.
Therefore $I_\pi$ and $I_\lambda$ cannot both satisfy Turaev's criterion,
and so $\pi$ is not the group of a non-orientable $PD_3$-complex.
Thus $X$ must be orientable.
\end{proof}

Our final lemma is needed to cope with three exceptional cases.

\begin{lemma}
Let $G=H\rtimes\mathbb{Z}/2\mathbb{Z}$, where $H=T_1^*,O_1^*$ or $I^*$.
Suppose that every element of $G$ divisible by $4$ is in $H$.
Then $G$ has a subgroup $W$ of order $6$ such that $[W:W\cap{H}]=2$.
\end{lemma}

\begin{proof}
Let $g$ be an element of order 2 whose image generates $G/H$.

Suppose first that $H=T_1^*$, with presentation
\[
\langle{x,y,z}|{x^2=(xy)^2=y^2}, ~z^3=1,~zxz^{-1}=y,~zyz^{-1}=xy\rangle.
\]
Then $\zeta{T_1^*}=\langle{x^2}\rangle$ has order 2.
The outer automorphism group $Out(T_1^*)$ is generated 
by the class of the involution $\rho$ which sends $x,y$ and $z$ 
to $y^{-1},x^{-1}$ and $z^2$, respectively.
(See page 221 of \cite{Hi}.)
Hence $\rho$ preserves the subgroup $S$ of order 3 generated by $z$.

If conjugation by $g$ induces an inner automorphism of $T_1^*$
there is an $h\in{T_1^*}$ such that $gxg^{-1}=hxh^{-1}$ for all $x\in{T_1^*}$.
Then $gh=hg$ and $h^2$ is central in $T_1^*$, so $(h^{-1}g)^2=h^2$ has order dividing 4.
Therefore $h^{-1}g$ has order 2, by hypothesis.

Otherwise we may assume that there is an $h\in{G^+}$ 
such that $gxg^{-1}=h\rho(x)h^{-1}$ for all $x\in{T_1^*}$, 
and so $\rho$ is conjugation by $h^{-1}g$.
Since $\rho$ is an involution $(h^{-1}g)^2$ is central in $T_1^*$.
We again see that $h^{-1}g$ has order 2.
In each case $h^{-1}g$ normalizes $S$, so
the subgroup  $W$ generated by $S$ and $h^{-1}g$ has order $6$,
while $h^{-1}g\not\in{H}$, so $[W:W\cap{H}]=2$.

The commutator subgroup of $O_1^*$ is $T_1^*$.
Since this is a characteristic subgroup, it is preserved by $g$.
The group $T_1^*$ is a non-normal subgroup of $I^*$,  of index 5.
Since $g$ acts as an involution on the set of conjugates of $T_1^*$ 
we may assume that it preserves $T_1^*$.
In each case the lemma follows easily from its validity for $H=T_1^*$.
\end{proof}

We may now give our main result.

\begin{theorem}
Let $X$ be an indecomposable, non-orientable $PD_3$-complex
such that $\pi=\pi_1(X)$ has infinitely many ends.
Then 
\begin{enumerate}
\item
$\pi\cong\pi\mathcal{G}$ where $(\mathcal{G},\Gamma)$ 
is an admissible graph of groups with all vertex groups one-ended and all edge groups $\mathbb{Z}/2\mathbb{Z}^-$;
\item$\pi\cong\pi^+\rtimes\mathbb{Z}/2\mathbb{Z}^-$;
\item $\pi^+\cong{G}*H$, where $G$ is a nontrivial free product of $PD_3$-groups
and $H$ is  free.
In particular, $\pi^+$ is torsion free.
\end{enumerate}
\end{theorem}

\begin{proof}
Let $\pi\cong\pi\mathcal{G}$, where $(\mathcal{G},\Gamma)$ 
is an admissible graph of groups.
At least one vertex group is infinite,
for otherwise $\pi$ has two ends,
by Theorems 7.1 and 7.4 of \cite{Hi12}.
Hence $\pi^+\cong{G}*H$, 
where $G$ is a nontrivial free product of $PD_3$-groups
and $H$ is virtually free.
Therefore $\pi^+$ is virtually torsion free.
Let $\kappa$ be the intersection of the conjugates in $\pi$ 
of a torsion free subgroup of finite index in $\pi^+$,
and let $\phi:\pi\to\pi/\kappa$ be the canonical projection.
Then $\kappa$ is orientable, torsion free and of finite index,
and $w$ factors through $\pi/\kappa$.

If $F$ is a finite subgroup then $\phi|_F$ is injective, 
and $\phi^{-1}(\phi(F))$ has finite index in $\pi$. 
Hence $\phi^{-1}(\phi(F))$ has a graph of groups structure in which 
all finite vertex groups are isomorphic to subgroups of $F$.
In particular, if $F$ is a non-orientable 2-group then
at least one of these vertex groups is a non-orientable 2-group,
and so there is a $g\in{F}$ such that $g^2=1$ and $w(g)=-1$, 
by part (3) of Lemma 2. 
Hence if, moreover, $F$ is cyclic then it has order 2.

Assume that there is a non-orientable finite vertex group $G_v$.
Then $G_v$ has a non-orientable Sylow 2-subgroup $S(2)$, 
and so there is a $g\in{S(2)}$ such that $g^2=1$ and $w(g)=-1$.
The orientable subgroup $G_v^+$ has periodic cohomology, 
with period dividing 4,
by Theorems 4.3 and 4.6 of \cite{Hi12}.
Moreover every element of $G_v$ divisible by $4$ is in $G_v^+$,
by the argument of the previous paragraph.

Let $g$ be an element of order 2 whose image generates $G_v/G_v^+$.
We may assume that $G_v^+\cong{B}\times\mathbb{Z}/d\mathbb{Z}$, 
where $B$ is either $\mathbb{Z}/a\mathbb{Z}\rtimes{Q(2^i)}$ (with $a$ odd
and $i\geq3$), 
 $T^*_k$ or $O^*_k$ (for some $k\geq1$), $I^*$ 
or $\mathbb{Z}/a\mathbb{Z}\rtimes_{-1}{\mathbb{Z}/2^e\mathbb{Z}}$
(with $a$ odd and $e\geq1$),
as in  the penultimate paragraph of \S1 above.
Suppose first that $G_v^+$ is not a 2-group. 
Then it has a nontrivial subgroup $S$ of order $p$, for some odd prime $p$.
If $d>1$ we may assume that $p$ divides $d$,
and then $S$ is characteristic in $G_v^+$.
This is also the case if $G_v^+\cong\mathbb{Z}/a\mathbb{Z}\rtimes{Q(8)}$ 
or $\mathbb{Z}/a\mathbb{Z}\rtimes_{-1}\mathbb{Z}/2^e\mathbb{Z}$
with $a$ odd (so $p$ divides $a$), 
or $G_v^+\cong{T^*_k}$ or $O^*_k$ with $k>1$ (so $p=3$).
In these cases $S$ is normalized by $g$, 
and the subgroup $H$ generated by $S$ and $g$ has order $2p$.
The remaining possibilities are that  $G_v^+\cong{T_1^*\times\mathbb{Z}/d\mathbb{Z}}$,
$O_1^*\times\mathbb{Z}/d\mathbb{Z}$ or $I^*\times\mathbb{Z}/d\mathbb{Z}$.
For these cases we appeal to Lemma 5, to see that $G_v$
has a non-orientable  subgroup $W$ of order $2p$.

Since $\phi^{-1}\phi(W)$ has finite index in $\pi$, 
it is again the group of a non-orientable $PD_3$-complex.
This complex has an indecomposable factor whose group has $W$ 
as one of its finite vertex groups, and so has
fundamental group $\kappa\rtimes{W}$.
But this factor is non-orientable, and so 
contradicts Lemma 3.

Therefore we may assume that $G_v^+$ is a 2-group.
If $S(2)^+\not=1$ (i.e., if $G_v^+$ is a nontrivial 2-group) 
it is cyclic or generalized quaternionic,
and so has an unique central element of order 2.
(Cf. Lemma 2.1 of \cite{Hi12}.)
Hence $G_v$ has a finite index subgroup $W\cong\mathbb{Z}/2\mathbb{Z}\times\mathbb{Z}/2\mathbb{Z}^-$. 
As before, passage to $\phi^{-1}\phi(W)$ leads to a contradiction, by Lemma 4.

Therefore all finite vertex groups are orientable.
But the graph $\Gamma$ is connected, and any edge connecting a finite vertex group 
to an infinite vertex group must be non-orientable, as in Lemma 2.
Since there is at least one infinite vertex group
there can be no finite vertex groups.

The second assertion follows from part (2) of Lemma 2,
and $\pi^+=\pi\mathcal{G}^+$
is the fundamental group of a graph of groups $(\mathcal{G}^+,\Gamma)$ with the same underlying graph $\Gamma$,
trivial edge groups and vertex groups $G_v^+$ all $PD_3$-groups.
Hence $\pi^+$ is torsion free, but not free.
\end{proof}

As observed at the end of \S2, when $X$ is a 3-manifold and  
$(\mathcal{G},\Gamma)$ is an admissible graph of groups
such that  $\pi=\pi\mathcal{G}$, all vertices of $\Gamma$ have even valence.
Can this observation be extended to the case of $PD_3$-complexes?
Although there are indecomposable $PD_3$-complexes which are not homotopy
equivalent to 3-manifolds  \cite{Hi12,wall},
it remains possible that every indecomposable, non-orientable $PD_3$-complex is
homotopy equivalent to a 3-manifold.

Corollary 7.5 of \cite{Hi12}  follows  immediately from Crisp's Theorem and Theorem 6.
(The argument  in \cite{Hi12} assumed that $\pi$ is virtually free.)

\begin{cor}  
Let $X$ be a $PD_3$-complex and $g\in\pi=\pi_1(X)$ a nontrivial element of finite order.
If $C_\pi(g)$ is infinite then $g$ has order $2$ and is orientation-reversing, 
and $C_\pi(g)=\langle{g}\rangle\times\mathbb{Z}$.
\qed   
\end{cor}

Are there any examples other than $RP^2\times{S^1}$ of indecomposable $PD_3$-complexes 
whose groups have a central element of order 2 with infinite centralizer?


\begin{thebibliography}{99}

\bibitem{AM} Adem, A. and Milgram, R.J. \textit{Cohomology of Finite Groups},

Grundlehren der Math. Wissenschaft vol. 309 (second edition),

Springer-Verlag, Berlin -- Heidelberg -- New York (2004).

\bibitem{Br} Brown, K.S. \textit{Cohomology of Groups},

Graduate Texts in Mathematics 87,

Springer-Verlag, Berlin -- Heidelberg -- New York (1982).

\bibitem{Cr} Crisp, J.S. The decomposition of Poincar\'e duality complexes,

Comment. Math. Helv. 75 (2000), 232--246.

\bibitem{CR} Curtis, C. W. and Reiner, I. \textit{Representation Theory of Finite Groups 
and Associative Algebras}, J.Wiley, New York (1962).

\bibitem{DD}  Dicks, W. and Dunwoody, M.J., \textit{Groups Acting on Graphs},

Cambridge studies in advanced mathematics 17, 

Cambridge University Press, Cambridge - New York - Melbourne (1989).

\bibitem{Ep} Epstein, D.B.A. Projective planes in 3-manifolds,

Proc. London Math. Soc.  11 (1957), 469--484.

\bibitem{Gi} Gilbert, N.D. Presentations of the automorphism group of a free product,

Proc. London Math. Soc. 54 (1987), 115--140.

\bibitem{Hi} Hillman, J.A. \textit{Four-Manifolds, Geometries and Knots},

Geometry and Topology Monographs 5, 

Geometry and Topology Publications (2002). (Revisions 2007 and 2014).

\bibitem{Hi12} Hillman, J.A. Indecomposable $PD_3$-complexes,

Alg. Geom. Top. 12 (2012), 131--153.

\bibitem{rob} Robinson, D.J. {\it A Course in the Theory of Groups},

Graduate Texts in Mathematics 80,

Springer-Verlag, Berlin -- Heidelberg -- New York (1982).

\bibitem{tu} Turaev, V.G. Three dimensional Poincar\'e duality complexes: homotopy splitting and classification,
Mat. Sbornik 180 (1989), 809--830.

English translation: Math. USSR-Sbornik 67 (1990), 261-282.

\bibitem{wall} Wall, C.T.C.  Poincar\'e complexes: I,

Ann. Math. 86 (1967), 213--245.

\end{thebibliography}
\end{document}